\documentclass[fontsize=11pt,a4paper,DIV12]{scrartcl}

\usepackage{ifthen}

\usepackage[utf8]{inputenc} 
\usepackage{cmbright}
\usepackage{url}
\let\email=\url
\DeclareOldFontCommand{\bf}{\normalfont\bfseries}{\mathbf}

\usepackage{graphicx}
\def\n{\par\noindent}

\usepackage{amssymb}
\usepackage{amsmath}
\usepackage{amsthm}
\usepackage{amsopn}

\usepackage{xfrac}
\usepackage{enumerate}

\newtheoremstyle{meiner} 
    {4pt}{3pt}           
    {\itshape}           
    {}                   
    {\sffamily\bfseries} 
    {}                   
    { }                  
    {}                   
\theoremstyle{meiner}

\usepackage[capitalize]{cleveref}

  \newtheorem{theorem}{Theorem}
  \crefname{theorem}{Theorem}{Theorems}
  \Crefname{theorem}{Theorem}{Theorems}
  \newtheorem{proposition}{Proposition}
  \newtheorem{lemma}{Lemma}
  \newtheorem{corollary}{Corollary}
  
  \crefname{claim}{Claim}{Claims}
  \Crefname{claim}{Claim}{Claims}
  \newtheorem{conjecture}{Conjecture}
  \newtheorem{observation}{Observation}
  
  \newtheorem{remark}{Remark}
  \crefname{remark}{Remark}{Remarks}
  \Crefname{remark}{Remark}{Remarks}
  \crefname{question}{Question}{Questions}
  \Crefname{question}{Question}{Questions}
  \crefname{observation}{Observation}{Observations}
  \Crefname{observation}{Observation}{Observations}
  \newtheorem{example}{Example}
  \crefname{example}{Example}{Examples}
  \Crefname{example}{Example}{Examples}
\theoremstyle{definition}
  \newtheorem{definition}{Definition}

\def\PsFig#1#2#3{}
\def\SC{\scshape}
\title{The Star Dichromatic Number%
}
\author{
\parbox{7cm}{\center
{\SC Winfried Hochst\"attler
\\[3pt]
\normalsize
\email{winfried.hochstaettler@fernuni-hagen.de}\\
\small
        {Fakult\"at f\"ur Mathematik und Informatik\\
         FernUniversit\"at in Hagen}}}
\and\hspace{-0.3cm}
\parbox{7cm}{\center
{\SC Raphael Steiner
\\[3pt]
\normalsize
\email{steiner@math.tu-berlin.de}\\
\small
        {Institut f\"ur Mathematik\\
        MA 5-1 \\
         Technische Universit\"{a}t Berlin}}}
}

\begin{document}


\maketitle

\begin{abstract}
  We introduce a new notion of circular colourings for digraphs. The
  idea of this quantity, called \emph{star dichromatic number}
  $\vec{\chi}^\ast(D)$ of a digraph $D$, is to allow a finer
  subdivision of digraphs with the same dichromatic number into such
  which are ``easier'' or ``harder'' to colour by allowing fractional
  values. This is related to a coherent notion for the vertex
  arboricity of graphs introduced in \cite{fracvert} and resembles the
  concept of the \emph{star chromatic number} of graphs introduced by
  Vince in \cite{vince} in the framework of digraph colouring. After
  presenting basic properties of the new quantity, including range,
  simple classes of digraphs, general inequalities and its relation to
  integer counterparts as well as other concepts of fractional
  colouring, we compare our notion with the notion of circular
  colourings for digraphs introduced in \cite{bokal} and point out
  similarities as well as differences in certain situations. As it
  turns out, the star dichromatic number shares all positive
  characteristics with the circular dichromatic number of Bokal et
  al., but has the advantage that it depends on the strong
  components of the digraph only, while the addition of a dominating
  source raises the circular dichromatic number to the ceiling. We
  conclude with a discussion of the case of planar digraphs and point
  out some open problems.
\end{abstract}


\section{Introduction}
Digraphs and graphs in this paper are considered loopless, but are allowed to have multiple parallel and anti-parallel arcs between vertices. Digraphs without parallel or anti-parallel edges are referred to as \emph{simple}. We will refer to edges $e$ in graphs by $uw$ where $u,w$ are the end vertices of $e$, if this does not lead to confusion with parallel edges. Given an arc (or directed edge) $e$ of a digraph, we use $e=u \rightarrow w$ or equivalently $e=(u,w)$ to express that $e$ has tail $u$ and head $w$. This is not to be understood as a proper equality but as a statement on the edge $e$. Cycles and paths in graphs and directed cycles and paths digraphs are always considered without repeated vertices.

Fractional colourings of graphs were introduced by Vince in
\cite{vince}, where the concept of the \emph{star chromatic number},
nowadays also known as the \emph{circular chromatic number} of a
graph, made its first appearance. The original definition of Vince is
based on so-called $(k,d)$-colourings, where colours at adjacent
vertices are not only required to be distinct as usual but moreover
'far apart' in the following sense: For every $k \in \mathbb{N}$ and
elements $x,y \in \{0,...,k-1\}$, let $\text{dist}_k(x,y):=|(x-y)
\text{mod } k|_k,$ where $|a|_k:=\min\{|a|,|k-a|\},$ for all
$a=0,...,k-1$, denote the {\em circular $k$-distance between $x$ and
  $y$}. Then we define:

\begin{definition}[cf. \cite{vince}]
  Let $G$ be a graph and $(k,d) \in \mathbb{N}^2, k \ge d$. A
  {\em $(k,d)$-colouring of $G$} is an assignment $c:V(G)
  \rightarrow \{0,...,k-1\} \simeq \mathbb{Z}_k$ of colours to the
  vertices so that $\text{dist}_k(c(u),c(w)) \ge d$ whenever $u,w$ are
  adjacent.
\end{definition}
Fixing a graph $G$, Vince furthermore considered the smallest possible value of $\frac{k}{d}$ where $(k,d)$ allows a legal colouring of $G$ as a fractional measure of the ``colourability'' of $G$.
\begin{definition}
Let $G$ be a graph. The quantity 
$$\chi^\ast(G):=\inf\left\{\frac{k}{d}\bigg\vert\exists (k,d) \text{-colouring of }G\right\} \in \mathbb{R}_+$$
is called the \emph{star chromatic number} resp. \emph{circular chromatic number} of $G$.
\end{definition}
The following theorem captures the most important elementary properties of the star chromatic number.
\begin{theorem}[cf. \cite{vince}] \label{FracPropUndir}
Let $G$ be a graph. Then the following holds:
\begin{enumerate}
\item[(i)] $\chi^\ast(G)$ is a positive rational number, and $\chi^\ast(G) \ge 2$ whenever $E(G) \neq \emptyset$ (otherwise $\chi^\ast(G)=1$). 
\item[(ii)] $\lceil \chi^\ast(G) \rceil=\chi(G)$, i.e., $\chi^\ast(G) \in (\chi(G)-1,\chi(G)]$.
\item[(iii)] For each rational number $q \in \mathbb{Q}, q=\frac{m}{n} \ge 2$, there is a graph $G_{m,n}$ with $\chi^\ast(G_{m,n})=\frac{m}{n}=q$.
\item[(iv)] For every $k,d \in \mathbb{N}$, there exists a $(k,d)$-colouring of $G$ if and only if $\frac{k}{d} \ge \chi^\ast(G)$.
\item[(v)] If $\chi^\ast(G)=\frac{m}{n}$, then there is a $(k,d)$-colouring of $G$ with $\frac{k}{d}=\frac{m}{n}$ and $k \leq |V(G)|$.
\end{enumerate}
\end{theorem}
For further details concerning circular chromatic numbers of graphs we
refer to the survey article \cite{xuding}.  

A first definition of circular colourings for digraphs was given by
Bokal et al.\ in \cite{bokal}, leading to the notion of the circular
dichromatic number of digraphs and graphs. Instead of $(k,d)$-pairs as
in the case of Vince, they, equivalently, use real numbers for their
definition: Given a $p \ge 1$, consider a plane-circle $S_p$ of
perimeter $p$ and define a \emph{strong circular $p$-colouring} of $D$
to be an assignment $c:V(D) \rightarrow S_p$ of (colouring) points on
$S_p$ to the vertices, in such a way that for every edge $e=(u,w)$ in
$D$, the one-sided distance of $c(u),c(w)$ (i.e., the length of a
clockwise arc connecting $u$ to $w$ in $S_p$) is at least $1$. More
formally, we can identify $S_p$ with the set $\mathbb{R}/p\mathbb{Z}$
and require that the unique representative of $c(w)-c(u) \in
\mathbb{R}/p\mathbb{Z}$ in the interval $[0,p)$, denoted by
$(c(w)-c(u)) \text{ mod }p$ is at least one. In this representation,
the clockwise direction on $S_p$ is identified with the positive
direction in $\mathbb{R}/p\mathbb{Z}$.  Since the notion of a strong
circular $p$-colouring turns out to be much less flexible, the authors
also define so-called \emph{weak circular $p$-colourings} of $D$, $p
\in [1,\infty)$, as maps $c:V(D) \rightarrow S_p$, such that equal
colours at both ends of an edge, i.e., $c(u)=c(w)$ where $e=(u,w) \in
E(D)$, are allowed, but at the same time, the one-sided distance of
$c(u),c(w)$ on $S_p$ is at least $1$ whenever they are distinct.
Moreover, each so-called colour class, i.e., $c^{-1}(t), t \in S_p$
has to induce an acyclic subdigraph of $D$. This seems much more
intuitive and closer to the definition of legal digraph colourings. \n
The {\em circular dichromatic number $\vec{\chi}_c(D)$} now is defined
as the infimum over all real values $p \ge 1$ for which $D$ admits a
strong circular $p$-colouring, or, equivalently (as shown in their
paper), as the infimum over all values $p \ge 1$ providing weak
circular $p$-colourings of $D$. Moreover, in the case of weak circular
$p$-colourings the infimum is always attained.

\begin{proposition}[\cite{bokal}] \label{circdic}
Let $D$ be a digraph. The real value
\begin{eqnarray*}
\vec{\chi}_c(D)&:=&\inf\{p \ge 1|\exists \text{ weak circular }p\text{-colouring of } D\}\\&=&\inf\{p \ge 1|\exists \text{ strong circular }p\text{-colouring of } D\}
\end{eqnarray*}
is called the \emph{circular dichromatic number} of $D$.
Furthermore,every digraph admits a weak circular
$\vec{\chi}_c(D)$-colouring. If $G$ is a graph and $\mathcal{O}(G)$
the set of its orientations, then we define the maximum
$$\vec{\chi}_c(G):=\max_{D \in \mathcal{O}(G)}{\vec{\chi}_c(D)}$$ to be the \emph{circular dichromatic number} of the graph $G$.
\end{proposition}

The following sums up the most basic properties of this quantity:

\begin{theorem}[\cite{bokal}] \label{PropertiesCircularDi}
Let $D$ be a digraph. Then the following holds:
\begin{enumerate}
\item[(i)] $\vec{\chi}_c(D) \ge 1$ is a rational number with numerator at most $|V(D)|$.
\item[(ii)] $\lceil \vec{\chi}_c(D) \rceil=\vec{\chi}(D)$, i.e., $\vec{\chi}_c(D) \in (\vec{\chi}(D)-1,\vec{\chi}(D)]$.
\item[(iii)] $\vec{\chi}_c(\cdot)$ attains exactly the rational numbers $q \in \mathbb{Q}, q \ge 1$.
\end{enumerate}
\end{theorem}
It was furthermore pointed out in \cite{mastert} that the following discrete notion of \emph{circular $(k,d)$-colourings} corresponds to the above notion of weak circular $p$-colourings:
\begin{definition} \label{DefCirc}
Let $D$ be a digraph and $k \ge d$ natural numbers. A \emph{circular $(k,d)$-colouring} is a vertex-colouring $c:V(D) \rightarrow \{0,...,k-1\} \simeq \mathbb{Z}_k$ such that $(c(w)-c(u)) \text{ mod }k \ge d$ or $c(u)=c(w)$ for all $e=(u,w) \in E(D)$ and each colour class $c^{-1}(i), i \in \mathbb{Z}_k$ induces an acyclic subdigraph of $D$.
\end{definition}

\begin{proposition}[\cite{mastert}]
  For every digraph $D$ and every $p \ge 1$, there exists a weak
  circular $p$-colouring of $D$ if and only if there is a circular
  $(k,d)$-colouring of $D$ for every pair $(k,d) \in \mathbb{N}^2$
  with $\frac{k}{d} \ge p$. Thus,
$$\vec{\chi}_c(D)=\inf\left\{\frac{k}{d} \bigg\vert \exists \text{ circular }(k,d)\text{-colouring of }D \right\}.$$
\end{proposition}

\section{The Star Dichromatic Number, General Properties}
In this section we introduce a new concept of fractional digraph colouring. 

\begin{definition} \label{DefDir} Let $D$ be a digraph, $(k,d) \in
  \mathbb{N}^2, k \ge d$. An \emph{acyclic $(k,d)$-colouring} of $D$
  is an assignment $c:V(D) \rightarrow \mathbb{Z}_k$ of colours to the
  vertices such that for every $i \in \mathbb{Z}_k$, the pre-image of
  the cyclic interval $A_i:=\{i,i+1,...,i+d-1\} \subseteq
  \mathbb{Z}_k$ of colours, $c^{-1}(A_i) \subseteq V(D)$, induces an
  acyclic subdigraph of $D$.
\end{definition}

It will be handy to also have an equivalent formulation allowing real numbers ready, which deals with the circles $S_p, p \ge 1$:
\begin{definition}
Let $p \in \mathbb{R}, p \ge 1$. For $a,b \in [0,p)$, we denote by $(a,b)_p$ the open ``interval'' $(a,b)_p:=\{y \in [0,p)|0<(y-a) \text{ mod }p < (b-a) \text{ mod }p \}$. Analogous definitions apply for $[a,b]_p,[a,b)_p, (a,b]_p$. In each case, we call $(b-a) \text{ mod }p$ the \emph{length} of the respective interval. For each $x \in S_p$, denote by $|x|_p:=\min\{x,p-x\}$ its two-sided distance to $0$.
\end{definition}
\begin{definition}
  Let $D$ be a digraph and $p \ge 1$. An \emph{acyclic $p$-colouring}
  of $D$ is an assignment $c:V(D) \rightarrow [0,p) \simeq
  \mathbb{R}/p\mathbb{Z}$ of ``colours'' to the vertices, such that
  for every open interval $I=(a,b)_p$ of length $1$ within $[0,p)
  \simeq \mathbb{R}/p\mathbb{Z}$, the subdigraph induced by the
  vertices in $c^{-1}(I)$ is acyclic.  The \emph{star dichromatic
    number of $D$}  now is defined as the infimum over the numbers $p$
  for which $D$ admits an acyclic $p$-colouring:
$$\vec{\chi}^\ast(D)=\inf\{p \ge 1|\exists \text{acyclic }p\text{-colouring of }D\}.$$

\end{definition}
The following ensures that there always exists a
$\vec{\chi}^\ast(D)$-colouring of a digraph $D$:
\begin{proposition} \label{closed} Let $P:=\{p \ge 1|\exists \text{ acyclic
  }p\text{-colouring of }D\} \subseteq [1,\infty)$. Then $P$ is
  closed. Furthermore, $D$ admits an acyclic
  $\vec{\chi}^\ast(D)$-colouring.
\end{proposition}
\begin{proof}
  Since $P$ is bounded from below, the latter claim is a consequence
  of the former. Let $(p_n)_{n \in \mathbb{N}}$ be a sequence of
  elements of $P$ convergent to some $p \ge 1$. We have to show that
  $p \in P$. Clearly, we may assume $p_n>p$ for all $n \in
  \mathbb{N}$.  For given $n$ let $c_n':V(D) \rightarrow [0,p_n)$
  denote a feasible $p_n$-colouring of $D$.  Scaling by
  $\frac{p}{p_n}$ we derive maps $c_n:V(D) \rightarrow [0,p),
  x\mapsto \frac{p}{p_n}c_n'(x)$ with the property that for every
  open interval $I \subseteq \mathbb{R}/p\mathbb{Z}$ of length at most
  $\frac{p}{p_n}$ there is no directed cycle in the digraph induced by
  $c_n^{-1}(I)$.
  We may consider $(c_n)_{n \in \mathbb{N}}$ as a sequence of vectors
  in $S_p^{|V(D)|}$. Applying the Theorem of Heine-Borel to $(c_n)_{n
    \in \mathbb{N}}$ yields a convergent subsequence $(c_{n_l})_{l \in
    \mathbb{N}}$. Let $c:=\lim_{l \rightarrow \infty}{c_{n_l}}$. Then
  $c:V(D) \rightarrow [0,p)$.  We claim that $c$ defines an
  acyclic $p$-colouring of $D$. 

  Assume to the contrary there was a directed cycle $C$ in $D$ such
  that $c(V(C))$ is contained in an open interval $I=(a,b)_p \subseteq
  S_p \simeq [0,p)$ of length $1$. Since $c(V(C))$ is finite, there
  exists $0<\varepsilon<\frac{1}{2}$ such that $c(V(C)) \subseteq
  (a+\varepsilon,b-\varepsilon)_p \subseteq (a,b)_p$.  Since $D$ is
  finite, $(c_{n_l})_{l \in \mathbb{N}}$ is a sequence convergent in
  ${S_p^{|V(D)|}}$ and $\lim_{n \to \infty}p_n=p$, we may choose $N
  \in \mathbb{N}$ such that $|c_{N}(x) - c(x)|<\frac{\varepsilon}{2}$
  for all $x \in V(D)$ and $p_{N}<\frac{p}{1-\varepsilon}$.  Now,
  $c_N(V(C)) \subseteq (a+\frac{\varepsilon}{2},b -\frac{\varepsilon}{2})
  \text{ mod } p$. Hence, we have found a directed cycle in the inverse
  image of an open interval of length $1-\varepsilon < \frac{p}{p_N}$,
  contradicting the properties of $c_N$.

  As $D$ is considered loopless, $P \neq \emptyset$ and thus $P$ is closed and
  bounded from below, which implies that it admits a minimum.
\end{proof}

The following equivalence now makes the relation between the discrete notion and the real-number-notion of acyclic colourings of digraphs precise:

\begin{proposition} \label{zshg} Let $D$ be a digraph. Then for every real
  number $p \ge 1$, $D$ admits an acyclic $p$-colouring if and only if
  it admits an acyclic $(k,d)$-colouring for every $(k,d) \in
  \mathbb{N}^2$ fulfilling $\frac{k}{d} \ge p$. Consequently,
$$\vec{\chi}^\ast(D)=\inf\left\{\frac{k}{d}\bigg\vert\exists \text{acyclic } (k,d)\text{-colouring of }D\right\}.$$
\end{proposition}
\begin{proof}
Assume for the first implication there was an acyclic $p$-colouring $c:V(D) \rightarrow [0,p)$ of $D$ and let $(k,d) \in \mathbb{N}^2$ with $\frac{k}{d} \ge p$ be arbitrary. Define a colouring $c_{k,d}$ of the vertices by
$$\forall x \in V(D): c_{k,d}(x)=\left\lfloor \frac{k}{p}c(x) \right\rfloor \in \{0,....,k-1\}.$$
We claim that this defines an acyclic $(k,d)$-colouring of $D$: Assume
to the contrary there was a directed cycle $C$ within
$c_{k,d}^{-1}(A_i)$ for some $i \in \{0,....,k-1\} \simeq
\mathbb{Z}_k$. Then for all $x \in V(C)$, $$\left(\left\lfloor
    \frac{k}{p}c(x) \right\rfloor-i\right) \text{ mod }k \leq d-1
\Rightarrow \left(\frac{k}{p}c(x)-i\right) \text{ mod }k<d.$$
Consequently, $(c(x)-\frac{ip}{k}) \text{ mod
}p=\frac{p}{k}((\frac{k}{p}c(x)-i) \text{ mod }k) <\frac{p}{k/d} \leq
1$. Hence, $c(V(C)) \subseteq (\frac{ip}{k},\frac{ip}{k}+1)_p$,
contradicting the definition of an acyclic colouring. 

For the reverse implication, assume that $p \ge 1$ such that for every
pair $(k,d) \in \mathbb{N}^2$ with $\frac{k}{d} \ge p$, there is an
acyclic $(k,d)$- colouring $c_{(k,d)}:V(D) \rightarrow \{0,...,k-1\}$
of $D$.  Let $\left((k_n,d_n)\right)_{n \in \mathbb{N}}$ be some
sequence in $\mathbb{N}^2$ such that $p_n:=\frac{k_n}{d_n} \ge p$ for
all $n \in \mathbb{N},$ and $\lim_{n \to \infty}\frac{k_n}{d_n}=p$.
Let $c_n=c_{(k_n,d_n)}:V(D) \to \{0,\ldots,k_n-1\}$ denote
corresponding acyclic $(k_n,d_n)$-colourings of $D$. We define
$c_{p_n}:V(D) \rightarrow [0,p_n)$ by
$$x\mapsto \frac{p_n}{k_n}c_n(x) \in [0,p_n).$$ 
We claim that for every $n$ this defines an acyclic $p_n$-colouring.
Assume to the contrary there was a cyclic open subinterval
$(a,b)_{p_n} \subseteq [0,p_n)$ of length $1$ containing the colours
of a directed cycle $C$ in $D$, then for every $x \in V(C)$, we would
have $$0<\left(\frac{p_n}{k_n}c_n(x)-a\right) \text{ mod }p_n<1
\Leftrightarrow 0<(c_n(x)-d_na) \text{ mod }k_n<\frac{k_n}{p_n}=d_n$$
and thus, with $i:=\lceil d_na \rceil \text{ mod }k_n$, we get $0 \leq
(c_n(x)-i) \text{ mod }k_n \leq d_n-1$ for all $x \in V(C)$, implying
$c_n(V(C)) \subseteq A_i$. This contradicts $c_n$ being an acyclic
$(k_n,d_n)$-colouring and shows that indeed, $p_n \in P, n \ge 1$,
where again, $P$ denotes the set of $p \ge 1$ allowing an acyclic
colouring of $D$. Since $P$ is closed (Proposition \ref{closed}), we
finally deduce that $p=\lim_{n \rightarrow \infty}{p_n} \in P$, and
thus the claimed equivalence follows.
\end{proof}
Although theoretically, the definition of $\vec{\chi}^\ast(D)$ as the
infimum of the set $P$ of real numbers might include irrational values
of $\vec{\chi}^\ast(D)$, the following statement shows that due to the
conditions on acyclic $p$-colourings which are given in terms of
a finite object, namely $D$, $\vec{\chi}^\ast(D)$ only attains
rational numbers with a certain bound on the numerator. Analogous
statements hold for other notions of circular colourings.
\begin{theorem} \label{rational}
Let $D$ be a digraph, $n=|V(D)|$. Then $\vec{\chi}^\ast(D)$ is a rational number of the form $\frac{k}{d}$ with $1 \leq d \leq k \leq n$.
\end{theorem}
\begin{proof}
  Our proof follows the lines of the one given for the 
  same result for $\vec{\chi}_c(D)$ in \cite{bokal} resp.\
  \cite{bojanbok}.  

  Let in the following $p:=\vec{\chi}^\ast(D)$. We may assume $p>1$. For a given acyclic
  $p$-colouring $c:V(D) \rightarrow S_p \simeq [0,p)$ of $D$ we
  consider the digraph $D_1(c)$, defined over the vertex set $V(D)$
  where $(u,w) \in E(D_1(c))$ if and only if $(c(w)-c(u)) \text{ mod
  }p=1$. Let $v_0 \in V(D)$ be a fixed reference vertex, we may assume
  that $c(v_0)=0$. We will show that we can choose $c$ such that for
  every vertex $v \in V(D)$, there is a directed path from $v_0$ to
  $v$ in $D_1(c)$. For this purpose, let $c$ be an acyclic
  $p$-colouring maximal with respect to the cardinality of the set
  $S(c)$ of vertices reachable from $v_0$ via directed paths in
  $D_1(c)$. Assume for a contradiction that $S(c) \neq V(D)$.  For $s \in
  [0,p),$ we define
  $$c_s(v):=\begin{cases} c(v), & \text{if }v \in S(c) \cr (c(v)-s)
    \text{ mod }p & \text{if }v \notin S(c).
  \end{cases}$$ Note that for each $s \in [0,p)$ so that $c_s$ is an acyclic $p$-colouring, we have $S(c_s) \supseteq S(c)$, and due to the maximality of $c$, $S(c_s)=S(c)$. Now, choose $s^\ast$ maximal with the property, that
  for all $s<s^\ast$ $c_s$ is an acyclic
  $p$-colouring. The assumption $S(c) \ne V(D)$
  now implies that $0 < s^\ast < p$ and $c_{s^\ast + \varepsilon}$ is
  not an acyclic $p$-colouring for arbitrarily small values of $\varepsilon
  >0$.  Therefore, there must exist a closed interval $[a,b]_p\subseteq
  S_p$ of length $1$ such that $c_{s^\ast}^{-1}([a,b]_p)$
  contains the vertices of a directed cycle $C$ and such that there are $u,w \in V(C)$ with $c_{s^\ast}(u)=a,c_{s^\ast}(w)=b$ and $u \in S(c), w \notin S(c)$. But this implies that $S(c)\cup \{w\} \subseteq
  S(c_{s^\ast})$ contradicting the choice of $c$.

  We now consider the case that there exists a vertex $v \in V(D)
  \setminus \{v_0\}$ and two directed $v_0$-$v$-walks $P_1$ and $P_2$
  of lengths $\ell(P_1)> \ell(P_2)$ that visit at most one vertex
  (possibly $v$) twice. This includes the case, that there exists a
  directed cycle in $D_1(c)$. Since $c(v)=\ell(P_1) \text{ mod
  }p=\ell(P_2) \text{ mod }p$ there exists some $m\in \mathbb{N}$ such
  that $mp=\ell(P_1)-\ell(P_2)$. But clearly $0<\ell(P_1)-\ell(P_2)<n$
  and hence $p=\frac{\ell(P_1)-\ell(P_2)}{m}$ as required.

  Thus we may assume, that for all vertices in $v \in V(D)$ all
  directed $v_0$-$v$ paths have the same length, defining a map $f:V
  \to \mathbb{N}, v \mapsto \ell(P_v)$ and $f(v) \text{ mod }p=c(v)$
  for all $v \in V(D)$. We will show that this contradicts the
  minimality of $p$.  For that purpose choose $\delta>0$ such that
  $p-\delta>1$ and for each pair $u,w \in V(D)$ of vertices with
  $(f(w)-f(u)) \text{ mod }p>1$, we have $(f(w)-f(u)) \text{ mod
  }(p-\delta)>1$. We claim that $x \mapsto c_{-\delta}(x):=f(x) \text{
    mod }(p-\delta)$ defines an acyclic $(p-\delta)$-colouring of $D$.
  Assume to the contrary there was a directed cycle $C$ in $D$ such
  that its image under $c_{-\delta}$ is contained in a closed interval
  $[c_{-\delta}(u),c_{-\delta}(w)]_{p-\delta} \supseteq
  c_{-\delta}(V(C))$ of length $<1$, where $u,w \in V(C)$. Let in the
  following $x \in V(C)$ be arbitrary. Then
  $(c_{-\delta}(x)-c_{-\delta}(u)) \text{ mod }(p-\delta)=(f(x)-f(u))
  \text{ mod }(p-\delta)<1$ and thus $(u,x) \notin
  E(D_1(c))$ and  $(f(x)-f(u)) \text{ mod }p =(c(x)-c(u)) \text
  {mod }p \leq 1$. We conclude, $(c(x)-c(u)) \text{ mod }p<1$ and there exists
  $\varepsilon >0 $ such that $(c(x)-c(u)) \text{ mod
  }p<1-\varepsilon$ for all $x \in V(C)$, contradicting
  $c$ being an acyclic
  $p$-colouring. 

 The claim follows.
\end{proof}
\begin{corollary}
  For a digraph $D$ we have $\vec{\chi}^\ast(D) \ge 1$ with equality
  if and only if $D$ is acyclic.
\end{corollary}
\begin{proof}
  The inequality holds by definition.  $\vec{\chi}^\ast(D) = 1$
  implies the existence of an acyclic $1$-colouring of $D$, and thus,
  since $V(D)$ is finite, that $D$ is acyclic.
\end{proof}
The following describes the relationship of $\vec{\chi}^\ast(D)$ with its integer counterpart.
\begin{theorem}
Let $D$ be a digraph. Then $\lceil \vec{\chi}^\ast(D)\rceil=\vec{\chi}(D)$, i.e., $\vec{\chi}(D)-1<\vec{\chi}^\ast(D) \leq \vec{\chi}(D)$.
\end{theorem}
\begin{proof}
  The latter inequality is an immediate consequence from Proposition
  \ref{zshg} and the fact that the acyclic $(k,1)$-colourings of $D$
  correspond exactly to legal $k$-digraph colourings of $D$ in the
  usual sense, for $k \in \mathbb{N}$. On the other hand let
  $p:=\vec{\chi}^\ast(D), k:=\lceil p \rceil \in \mathbb{N} $ and let
  $c:V(D) \rightarrow S_p$ denote an acyclic $p$-colouring of $D$.
  Since $p\le k$ and $V(D)$ is finite, we find $k$ pairwise disjoint
  cyclic subintervals $I_1,...,I_k$, each of length less than $1$ such
  that all $v \in V(D)$ are mapped to the interior of one of these.
  Thus $c^{-1}(I_i), i=1,...,k$ induces an acyclic subdigraph of $D$,
  this way defining a $k$-digraph colouring of $D$, proving $k
  \geq \vec{\chi}(D)$.
\end{proof}

\section{Relations to Other Fractional Digraph Colouring Parameters}
We briefly review the notions of the fractional chromatic
numbers of graphs and digraphs in order to draw a comparison with our
new fractional colouring number. The fractional dichromatic number
will be a main tool for deriving lower bounds on star dichromatic
numbers.

 \pagebreak[3]
\begin{definition}[cf.\ \cite{kneser} and \cite{frdichr}] \noindent
\begin{itemize}
\item[(A)] Let $G$ be a graph. Denote by $\mathcal{I}(G)$ the collection of independent vertex subsets of $G$, and for each $v \in V(D)$, let $\mathcal{I}(G,v) \subseteq \mathcal{I}(G)$ be the subset containing only those sets including~$v$.  The \emph{fractional chromatic number} $\chi_f(G)$ of $D$ is now defined as the value of the following linear program
\begin{align}
&\min \sum_{I \in \mathcal{I}(G)}{x_I}\\ 
\text{ subj.\ to }&
\sum_{I \in \mathcal{I}(G,v)}{x_I} \ge 1, \text{ for all } v \in V(G) \nonumber\\
& x \ge 0.\nonumber
\end{align}
\item[(B)] 
  Let $D$ be a digraph. Denote by $\mathcal{A}(D)$ the collection of
  vertex subsets of $D$ inducing an acyclic subdigraph, and for each
  $v \in V(D)$, let $\mathcal{A}(D,v) \subseteq \mathcal{A}(D)$ be the
  subset containing only those sets including $v$.  The
  \emph{fractional dichromatic number} $\vec{\chi}_f(D)$ of $D$ is now
  defined as the value of 
 \begin{align} \label{primal}
&\min \sum_{A \in \mathcal{A}(D)}{x_A}\\
\text{  subj.\ to }&
\sum_{A \in \mathcal{A}(D,v)}{x_A} \ge 1, \text{ for all } v \in V(D) \cr
& x \ge 0.\nonumber
\end{align}
For a graph $G$, we define $\vec{\chi}_f(G):=\max_{\mathcal{O}(G) \text{ orient.}}{\vec{\chi}_f(\mathcal{O}(G))}$ to be its \emph{fractional dichromatic number}.
\end{itemize}
\end{definition}
The following inequality chain finally describes the behaviour of the
three notions of fractional digraph colouring numbers introduced so
far in general and shows that the star dichromatic number separates
the fractional from the circular chromatic number.
\begin{theorem}\label{3Inequ}
Let $D$ be a digraph. Then $\vec{\chi}_f(D) \leq \vec{\chi}^\ast(D) \leq \vec{\chi}_c(D)$.
\end{theorem}
\begin{proof}
  Let $\vec{\chi}^\ast(D)=\frac{k}{d}$ and $c_k:V(D) \rightarrow
  \mathbb{Z}_k$ be an acyclic $(k,d)$-colouring for two integers $0<d\le
  k$. Given $A \in \mathcal{A}(D)$ let $i_A :=|\{i \in \mathbb{Z}_k
  \mid A= c_k^{-1}(\{i,...,i+d-1\})\}|$ and define
  $x_A=\frac{i_A}{d}$.  Then for every vertex $v \in V(D)$, we have
$$\sum_{A \in \mathcal{A}(D,v)}{x_A}=\sum_{i \in \mathbb{Z}_k:c_k(v) \in \{i,...,i+d-1\}}{\frac{1}{d}}=1$$
Hence, $x$ is feasible for the above linear programm implying
$\vec{\chi}_f(D)\le \sum_{i \in
  \mathbb{Z}_k}{\frac{1}{d}}=\frac{k}{d}= \vec{\chi}^\ast(D)$. 

For the second inequality, it suffices to show that for every $p \ge
1$, any weak circular $p$-colouring $c:V(D) \rightarrow [0,p)$ in the
sense of Bokal et al.\ is also an acyclic $p$-colouring of $D$.
Assume to the contrary there was a directed cycle $C$ in $D$ such
that $c(V(C))$ is contained in an open subinterval of length $1$ in
$S_p \simeq [0,p)$. We may assume $c(V(C)) \subseteq (0,1)_p$. Then
obviously, $0<(c(w)-c(u)) \text{ mod }p <1$ for every edge $(u,w) \in
E(C)$ with $c(w)>c(u)$, contradicting the definition of weak circular
colourings.  Thus $c(C)$ consists of a single point $\{t\} \subseteq
S_p$, which means that $c^{-1}(t)$ is not acyclic, a contradiction. Hence $c$ is a weak colouring
and the claim follows.
\end{proof}
It is well-known that for symmetric orientations of graphs  the chromatic number of the original graph equals their dichromatic number. Similar relations hold for fractional, star and circular dichromatic number.
\begin{remark}\label{remark:symmetric}
Let $G$ be an undirected graph, and denote by $S(G)$ its symmetric orientation where every undirected egde in $E(G)$ is replaced by an anti-parallel pair of arcs. Then 
$$\vec{\chi}_f(S(G))=\chi_f(G), \vec{\chi}^\ast(S(G))=\vec{\chi}_c(S(G))=\chi^\ast(G).$$
\end{remark}
\begin{proof}
  The first equality follows from the fact that the vertex subsets in
  $S(G)$ inducing acyclic subdigraphs are exactly the independent
  vertex sets in $G$. Furthermore, since every parallel replacement
  pair of arcs gives rise to a directed 2-cycle, weak circular
  $p$-colourings as well as $p$-colourings according to our definition
  of $S(G)$, for every $p \ge 1$, are exactly those maps $c:V(G)
  \rightarrow [0,p)$ with $\text{dist}_p(c(u),c(w)) \ge 1$ for every
  adjacent pair of vertices $u,w$, implying the latter
  two equalities.
\end{proof}
As we will see in the next section, when dealing with planar digraphs,
finding digraphs without large acyclic vertex subsets yields good
lower bounds for the fractional and thus also the star dichromatic
number. This is made precise by the following inequality.
\begin{lemma} \label{maxsize}
Let $D$ be a digraph and denote by $\vec{\alpha}(D)$ the maximum size of a vertex subset of $D$ inducing an acyclic subdigraph. Then $\vec{\chi}_f(D) \ge \frac{|V(D)|}{\vec{\alpha}(D)}$.
\end{lemma}
\begin{proof}
Consider the dual of the linear program \eqref{primal},
\begin{align}\label{dual}
&\max \sum_{v \in V}{y_v} \\
\text{ subj. to } &
\sum_{v \in A}{y_v} \leq 1, & \text{ for all } A \in \mathcal{A}(D) \nonumber\\
&y \ge 0.\nonumber
\end{align}
Define $y_v:=\frac{1}{\vec{\alpha}(D)}$ for each vertex $v \in V$. The $y$ clearly is feasible for \eqref{dual} and the result follows by linear programming duality. 
\end{proof}
We now finally present a construction of digraphs (which are part of
the class of so-called \emph{circulant digraphs}) whose star
dichromatic numbers attain every rational number $q \ge 1$. The same
digraphs were used in \cite{bokal}.
\begin{theorem}
Let $(k,d) \in \mathbb{N}^2$ with $k \ge d$. Denote by $\vec{C}(k,d)$ the digraph defined over the vertex set $V(\vec{C}(k,d)):=\{0,...,k-1\} \simeq \mathbb{Z}_k$ so that each vertex $i \in \mathbb{Z}_k$ has exactly $k-d$ outgoing arcs, namely $(i,j), j=i+d,i+d+1,...,i+k-1$. Then
$$\vec{\chi}_f(D) = \vec{\chi}^\ast(D) = \vec{\chi}_c(D)= \frac{k}{d}.$$
Therefore, $\vec{\chi}^\ast(D)$ attains every rational number $q \geq 1$.
\end{theorem}
\begin{proof}

  According to Theorem \ref{3Inequ} it suffices to show that
  $\frac{k}{d} \leq \vec{\chi}_f(\vec{C}(k,d))$ and
  $\vec{\chi}_c(\vec{C}(k,d)) \leq \frac{k}{d}$.  

Let $A \in
  \mathcal{A}(\vec{C}(k,d))$, then $\vec{C}(k,d)[A]$, being acyclic
  contains a sink $a \in A \subseteq \mathbb{Z}_k$ and therefore
  $A \cap \{a+d,...,a+k-1\}= \emptyset$, proving $|A| \le d$ and the first inequality follows using Lemma \ref{maxsize}.

  For the  other  inequality,  note that  $c_{k,d}(i):=\frac{i}{d} \in
  [0,\frac{k}{d})$  for all $i  \in  V(\vec{C}(k,d))$ defines a strong
  $\frac{k}{d}$-colouring $c_{k,d}$ of $D$.

\end{proof}
Putting $k=n, d=n-1$ in the above, we immediately get the following.
\begin{corollary} \label{circles}
For every $n \in \mathbb{N}$,
$$\vec{\chi}_f(\vec{C}_n)=\vec{\chi}^\ast(\vec{C}_n)=\vec{\chi}_c(\vec{C}_n)=\frac{n}{n-1}.$$
\end{corollary}
While the above provides examples for digraphs where the three different concepts of fractional digraph colouring coincide, we now focus on constructing examples of digraphs where the numbers vary significantly in order to point out differences of the approaches. \n
First of all, it is well-known that contrary to the star chromatic number, the fractional chromatic number of a graph does not fulfil a ceiling-property, but can be arbitrarily far apart from the chromatic number of the graph. As a consequence we conclude that circular and star dichromatic number can be arbitrarily far apart of the fractional dichromatic number:
\begin{theorem}
  For every $C \in \mathbb{R}_+$, there is a digraph $D$ with
  $\vec{\chi}^\ast(D)-\vec{\chi}_f(D) =
  \vec{\chi}_c(D)-\vec{\chi}_f(D) \ge C$.
\end{theorem}
\begin{proof}
  By Remark~\ref{remark:symmetric} the result follows from the same
  observation for undirected graphs. As is well known for the Kneser
  graphs $G_n:=K(n,2), n \ge 4$, we have $\chi(G_n)-\chi
  _f(G_n)=(n-2)-\frac{n}{2}=\frac{n}{2}-2 \rightarrow \infty$ (cf.\
  \cite{kneser}, page 32).
\end{proof}

Now we compare $\vec{\chi}^\ast$ and $\vec{\chi}_c$ in more detail. We
see the main advantage of our new parameter in the fact, that it is
sufficient to consider only the strong components of a digraph $D$ in
order to compute $\vec{\chi}^\ast(D)$. 
\begin{observation} \label{CutVerbinden}
Let $D$ be a digraph and $S=D(X,\overline{X})$ a directed cut. Let $D_1:=D[X], D_2:=D[\overline{X}]$. Then 
$$\vec{\chi}(D)=\max\{\vec{\chi}(D_1),\vec{\chi}(D_2)\}, \vec{\chi}^\ast(D)=\max\{\vec{\chi}^\ast(D_1),\vec{\chi}^\ast(D_2)\}.$$
\end{observation}

On the other hand for the circular dichromatic number the existence of
a dominating source completely destroys any extra information we hope to 
gain compared to the dichromatic number.
\begin{proposition}\label{prop:dominating}
  Let $D$ be a digraph. We denote by $D^s$ the digraph arising from
  $D$ by adding an extra vertex $s$, which is a source adjacent to
  every vertex in $V(D)$. Then $\vec{\chi}_c(D^s)=\vec{\chi}(D)$.
\end{proposition}
\begin{proof}
  By Observation \ref{CutVerbinden} we have $\vec{\chi}_c(D^s) \leq
  \vec{\chi}(D^s)=\vec{\chi}(D)$. Assume contrary to the assertion
  that there was a strong $p$-colouring $c$ of $D^s$ with
  $p<\vec{\chi}(D)=:k$. We may assume $c(s)=0$. According to the
  definition of a strong colouring, the interval $[0,1)_p$ does not
  contain any other vertices, hence $c(V(D)) \subseteq [1,p)_p$. Since
  $p-1<k-1$, we can decompose the interval $[1,p)_p$ into $k-1$
  pairwise disjoint cyclic subintervals $I_1,...,I_{k-1}$ of $S_p$,
  each of length less than one and covering all the finitely many
  colouring points. If $(u,w)$ is an edge such that $c(u),c(w) \in I_l$
  are contained in the same interval, then, since $c$ is a strong
  colouring, we must have $c(u)>c(w)$. Hence, each $c^{-1}(I_l)$ induces an
  acyclic subdigraph of $D$ for each $l$, all together defining a
  $(k-1)$-digraph colouring of $D$, contradiction.  This proves the
  claim.
\end{proof}

\begin{example} \label{ex}
$\vec{\chi}^\ast(\vec{C}_n^s)=\frac{n}{n-1}<2=\vec{\chi}_c(\vec{C}_n^s)$ for all $n \ge 3$.
\end{example}
\begin{proof}

  According to Observation \ref{CutVerbinden} and Theorem
  \ref{circles} we have $\vec{\chi}^\ast(\vec{C}_n^s) =
  \vec{\chi}^\ast(\vec{C}_n)=\frac{n}{n-1}$. The remaining equality follows
  immediately from Proposition~\ref{prop:dominating}.
\end{proof}
\begin{figure}
\centering
\includegraphics[scale=1]{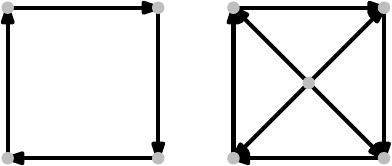}
\caption{Stacking a source into a directed 4-cycle. While the star dichromatic number remains unchanged with value $\frac{4}{3}$, the circular dichromatic number jumps from $\vec{\chi}_c(\vec{C}_4)=\frac{4}{3}$ to $\vec{\chi}_c(\vec{C}_4^s)=2$.} \label{DQ}
\end{figure}

\begin{corollary}
For every $\varepsilon>0$ there is a digraph $D$ with $\vec{\chi}_c(D)-\vec{\chi}^\ast(D) \ge 1-\varepsilon$.
\end{corollary}

\section{The Star Dichromatic Number of Simple Planar Digraphs and Circular Vertex Arboricity} \label{va}
Let $G$ be a given (unoriented) graph. If we want to estimate $\vec{\chi}^\ast(G)$, we need to find $(k,d)$-digraph colourings for every possible orientation of $G$. The simplest way of doing this is to find a single colouring of $V(G)$ yielding a legal $(k,d)$-colouring on all the possible orientations at the same time. This leads to the following definition which is introduced in \cite{fracvert}
\begin{definition}[\cite{fracvert}] \label{DefVa}
Let $G$ be a graph and $(k,d) \in \mathbb{N}^2, k \ge d$. A $(k,d)$-tree-colouring of $G$ is a colouring $c:V(G) \rightarrow \mathbb{Z}_k \simeq \{0,...,k-1\}$ of the vertices so that with $A_i:=\{i,i+1,...,i+d-1\} \subseteq \mathbb{Z}_k,$ $c^{-1}(A_i)$ induces an acyclic subgraph of $G$ for all $i \in  \mathbb{Z}_k$.
\end{definition}
The authors of \cite{fracvert} now define the {\em circular vertex arboricity} of a graph $G$ as the minimal value
$$\text{va}^\ast(G):=\inf\left\{\frac{k}{d}\bigg\vert \exists (k,d)\text{-colouring of }G\right\}.$$
The above now immediately implies
\begin{remark} \label{rem}
For every graph $G$, $\vec{\chi}^\ast(G) \leq \text{va}^\ast(G)$.
\end{remark}
As in the previous chapter, they also proved an alternative representation of this fractional quantity in terms of real numbers:
\begin{definition}[\cite{fracvert}]
Let $G$ be a graph and $p \ge 1$. Then a {\em $p$-circular tree colouring} of $G$ is defined as an assignment $c:V(G) \rightarrow S_p \simeq [0,p)$ so that for every open interval $I=(a,b)_p \subseteq [0,p)$ of length $1$
, $c^{-1}(I)$ induces an acyclic subgraph of $G$.
\end{definition}
\begin{theorem}[\cite{fracvert}]
For every graph $G$ we have
$$\text{va}^\ast(G)=\inf\{p|\exists p\text{-circular tree colouring of }G\}.$$
\end{theorem}
An important conjecture related to colourings of digraphs is the 2-colour-conjecture by Victor-Neumann-Lara:
\begin{conjecture}[Neumann-Lara, 1985]
$\vec{\chi}(D) \leq 2$ for every simple planar digraph $D$.
\end{conjecture}
According to the above, this is equivalent to $\vec{\chi}^\ast(D) \leq 2$ for simple planar digraphs. While the conjecture still remains unproven and since the best known general result so far only guarantees the existence of $3$-colourings of simple planar digraphs (via vertex arboricity,  \cite{chartrand2}), the following can be seen as an improvement of the upper bound $3$ for the star dichromatic number of planar digraphs:
\begin{theorem}
Let $D$ be a simple planar digraph. Then $\vec{\chi}^\ast(D) \leq 2.5$.
\end{theorem}
\begin{proof}
In \cite{fracvert} it is proved that $\text{va}^\ast(G) \leq 2.5$ for simple planar graphs. The claim now follows from Remark \ref{rem}.
\end{proof}

While the star dichromatic number can be considered an oriented version of
the circular vertex arboricity there does not seem to be an unoriented
counterpart to the circular colourings introduced by Bokal et al.. Note
that any map $c:V(G) \rightarrow S_p \simeq [0,p)$ that is a
simultaneous weak circular $p$-colouring of each possible orientation
of a graph $G$ which is no forest necessarily must have $p \ge 2$. 

The bound $\vec{\chi}^\ast(D) \leq 2$ for planar digraphs as a
consequence of the 2-colour-conjecture is best-possible as there exist
simple planar digraphs with star dichromatic number arbitrarily close
to $2$. This is a consequence of the case $g=3$ of the following
theorem.

\begin{theorem}
  For every $g \ge 3$ and every $\varepsilon>0$, there exists a planar
  digraph $D$ of digirth $g$ with $\vec{\chi}^\ast(D) \in
  [\frac{g-1}{g-2}-\varepsilon,\frac{g-1}{g-2}]$.
\end{theorem}
\begin{proof}
  Knauer et al.\ \cite{withoutlargeacyclic} constructed a sequence
  $(D_f^g)_{f \ge 1}$ of planar digraphs of digirth $g$ with
  $|V(D_f^g)|=f(g-1)+1$ and so that for the maximum order
  $\vec{\alpha}(D_f^g)$ of an induced acyclic subdigraph of $D_f^g$,
  we have $\vec{\alpha}(D_f^g) \leq \frac{|V(D_f^g)|(g-2)+1}{g-1}$ for
  all $f \ge 1$. Applying Lemma \ref{maxsize} this yields
  $$\vec{\chi}^\ast(D_f^g) \ge \vec{\chi}_f(D_f^g) \ge
  \frac{|V(D_f^g)|}{\vec{\alpha}(D_f^g)} \ge
  \frac{|V(D_f^g)|(g-1)}{|V(D_f^g)|(g-2)+1}.$$ Since the latter
  expression is convergent to $\frac{g-1}{g-2}$ for $f \rightarrow
  \infty$, it remains to show that all the $D_f^g, f \ge 1$ admit
  acyclic $(g-1,g-2)$-colourings.  This is easily seen using the
  inductive construction described in \cite{withoutlargeacyclic}.  For
  $f \ge 2$, $D_{f}^g$ arises from $D_{f-1}^g$ by adding an extra
  directed path $P=s_1,...,s_{g-1}$ with $g-1$ new vertices whose only
  connections to $V(D_f^g)$ consist of two vertices $x \neq y \in
  V(D_{f-1}^g)$ that both are adjacent to $x_1$ and $x_{g-1}$ via
  edges that are oriented in such a way that $x,P$ as well as $y,P$
  induce directed cycles. Now we inductively find an acyclic
  $(g-1,g-2)$-colouring by colouring the vertices of $P$ with the
  $g-1$ pairwise distinct colours. Clearly, this cannot create any new
  directed cycle using at most $g-2$ colours.
\end{proof}
There is some evidence that the construction given in
\cite{withoutlargeacyclic} is asymptotically best-possible. Thus, we
are tempted to  generalize the
2-colour-conjecture as follows:
\begin{conjecture} \label{girth} For every planar digraph $D$ of
  digirth at least $g \ge 3$, we have $\vec{\chi}^\ast(D) \leq
  \frac{g-1}{g-2}$. \n In other words, $D$ admits a colouring with
  $g-1$ colours so that each directed cycle in $D$ contains each
  colour at least once.
\end{conjecture}
The  above implies  that  this bound for a  given $g$,  if  true, is
best-possible. We furthermore note that these  upper bounds for $g \ge
4$ do not apply for the circular dichromatic number $\vec{\chi}_c(D)$.
(take e.g. $\vec{C}_g^s$ from example \ref{ex}).  We  are not aware of
an example for which  the bound in  the above inequality  is attained
with equality for $g \ge 4$. For the case $g=3$, we have the following (minimal) example of a simple planar digraph with star dichromatic number exactly $2$.
\begin{figure} 
\centering
\includegraphics[scale=1]{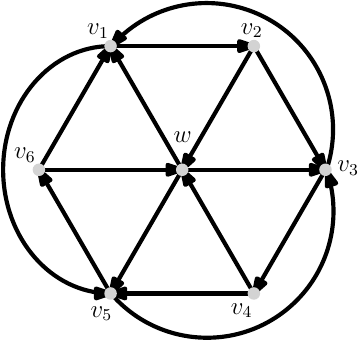}
\caption{The simple planar digraph $D_7$ with star dichromatic number exactly $2$.} \label{svsd}
\end{figure}
\begin{proposition}
The digraph $D_7$ depicted in Figure \ref{svsd} has $\vec{\chi}^\ast(D_7)=2$, while $\vec{\chi}^\ast(D) \leq \frac{5}{3}$ for any simple planar digraph $D$ on at most $5$ vertices. 
\end{proposition}
\begin{proof}
  The labels $w,v_1,\ldots,v_6$ of the vertices in $D_7$ refer to
  Figure \ref{svsd}.  To see that $\vec{\chi}^\ast(D_7) \leq 2$,
  notice that $c:V(D_7) \rightarrow \{0,1\}$,
  $c(w):=c(v_1):=c(v_3):=0, c(v_2):=c(v_4):=c(v_5):=c(v_6):=1$ defines
  a valid $2$-colouring of $D_7$. Assume now for a proof by
  contradiction that we had $\vec{\chi}^\ast(D_7)<2$ and thus
  (according to Theorem \ref{rational} and Proposition \ref{zshg})
  that there was an acyclic $(k,d)$-colouring of $D_7$ where $1 \leq d
  \leq k \leq 7$ are integers such that $\frac{k}{d} <2$. The latter
  implies $\frac{k}{d} \leq \frac{7}{4}$ and consequently the
  existence of an acyclic $(7,4)$-colouring $c_{7,4}:V(D_7)
  \rightarrow \mathbb{Z}_7 \simeq \{0,\ldots,6\}$ of $D_7$. Without
  loss of generality, we may assume that $c(w)=0$. Because any pair of
  elements is contained in a cyclic subinterval of length $4$ in
  $\mathbb{Z}_7$, the vertices of any directed triangle in $D_7$ must
  receive pairwise distinct colours. For any vertex $v_i$ is contained
  in a directed triangle together with $w$, we must have $c(v_i) \neq
  0, i=1,\ldots,6$. Considering the directed triangle $v_1v_5v_3$, we
  find that at least one of the vertices $v_1,v_3,v_5$ must have
  colour $1$ or $6$. Because of the symmetry of $D_7$ we may assume
  that $c(v_1) \in \{6,1\}$. Possibly after replacing $c$ with $-c$ we
  can even assume $c(v_1)=1$. Looking at the triangle $v_1,v_2,w$,
  this forces $c(v_2)=4$. As $v_1v_2v_3$ forms a directed triangle, it
  follows that $c(v_3) \in \{5,6\}$. Assume first that $c(v_3)=6$.
  Looking at the triangle $wv_3v_4$, this forces $c(v_4)=3$, and
  because $v_3v_4v_5$ is a directed triangle as well, this implies
  $c(v_5) \in \{1,2\}$. This now is a contradiction, because it means
  that the colour set of the directed triangle $v_1v_3v_5$ is
  contained in the cyclic subinterval $\{6,0,1,2\}$ of $\mathbb{Z}_7$.
  Consequently, we may assume that we are in the case of $c(v_3)=5$.
  The directed triangle $wv_3v_4$ now forces $c(v_4) \in \{2,3\}$, and
  as the colour set of $v_3v_4v_5$ must not be contained in
  $\{2,3,4,5\}$, $c(v_5)$ has to be either $0$, $1$ or $6$. This now
  leads to the contradiction that $\{c(v_1),c(v_3),c(v_5)\} \subseteq
  \{5,6,0,1\}$.

The second part of the claim can be verified by checking the existence of $(5,3)$-colourings of all simple planar digraphs on up to $5$ vertices. This is a simple but lengthy case distinction by hand but can be easily checked using a brute force program run on a standard personal computer.
\end{proof}
Naturally, one might expect that the $K_4$ is an extremal example for colourings of simple planar digraphs. Surprisingly, this is not the case. Considering more generally odd and even wheels we find:

\pagebreak[3]

\begin{example}
For $k\ge 3$ denote by $W_k$ the wheel with $k+1$ vertices. 
\begin{enumerate}
\item[(A)] If $k$ is odd, then
$$\vec{\chi}^\ast(W_k)=\vec{\chi}_f(W_k)= \frac{3}{2}.$$
\item[(B)] If $k$ is even, then 
$$\vec{\chi}^\ast(W_k) = \frac{5}{3} \quad \text{ but } \quad     
\vec{\chi}_f(W_k)= \frac{3k-2}{2k-2}. $$
\end{enumerate}
\end{example}

\begin{proof}
In the following, whenever we refer to a vertex $w$, it is to be understood as the dominating vertex of the respective wheel we deal with. In the following, wheels are considered to be canonically embedded in the plane such that $w$ is the only inner vertex. \par\noindent
\begin{itemize}
\item[(A)] As a wheel contains a triangle, using Corollary
  \ref{circles}, we find $\frac{3}{2}=\vec{\chi}_f(C_3) \leq
  \vec{\chi}_f(W_k) \leq \vec{\chi}^\ast(W_k).$ Next, we construct an
  acyclic $(3,2)$-coloring.  Since $k$ is odd, along the circular
  ordering of incoming and outgoing edges around $w$, there has to be
  a consecutive pair of edges with the same orientation, i.e., both
  incoming or both outgoing. Denote by $x_1,x_2$ their end vertices on
  the rim. We now color $W_k$ by assigning $0$ to $w$ and colouring
  the outer cycle using alternatingly $1$ and $2$ except for
  $x_1,x_2$, which both receive $1$. Clearly, this is an acyclic
  $(3,2)$-coloring, unless the outer cycle is directed. In that case,
  not only the triangle $wx_1x_2$, but also one of its neighbouring
  triangles is not directed. Hence, we may assume that, say, $x_1$ is
  not a vertex of a directed triangle. Now recoloring $x_1$ with color
  $0$ yields an acyclic $(3,2)$-coloring. Hence for every orientation
  $D$ of $W_k$ we find $\vec{\chi}^\ast(D) \leq \frac{3}{2}$.
\item[(B)] We first  prove  $\vec{\chi}^\ast(D) \leq \frac{5}{3}$  for
  all  orientations $D$ of  $W_k$. If  the  outer cycle is  undirected
  similar to case $(A)$ we find an acyclic $(3,2)$-colouring of $D$ by
  assigning $0$ to  the central vertex and  alternatingly $1,2$ to the
  outer vertices. Hence we may assume that the outer cycle is directed
  in $D$. If there exists a  pair of consecutive vertices $x_1,x_2$ on
  the    outer cycle where $wx_1x_2$   is   not  a directed  triangle,
  recoloring    either  $x_1$  or $x_2$  by    $0$   yields an acyclic
  $(3,2)$-coloring as  in the odd  case.  So, we   may assume that the
  outer  cycle is  directed  and   all   edges incident  to  $w$   are
  alternatingly incoming and outgoing in the cyclic order of $E_D(w)$.
  Hence, in  the cyclic order, every  second triangle is  directed and
  the others are not. 
%
  We define an acyclic $(5,3)$-colouring of $D$ by starting with a
  $0,2,3$-colouring of the vertices, where $w$ receives colour $0$ and
  the outer vertices alternating colours $2$ and $3$ such that the
  directed triangles have its vertices coloured by $0,3,2$ in cyclic
  order. Now choose one directed edge whose tail is coloured by $2$
  and recolour its head with $4$ and its tail with $1$. It is now
  easily seen that the vertices of no directed triangle nor of the
  outer cycle are contained in the union of three consecutive colour classes of
   colours of $\mathbb{Z}_5=\{0,1,2,3,4\}$, which
  proves $\vec{\chi}^\ast(D) \leq \frac{5}{3}$ also in this case.

  Next we show that $\vec{\chi}^\ast(W_k) \ge \frac{5}{3}$.  Clearly,
  this can be true only for the orientation considered the last for
  the upper bound.  Let $p$ be any real number admitting an acyclic
  colouring. As $D$ contains a directed triangle, $p \ge \frac{3}{2}$.
  Assume for a contradiction $ p<\frac{5}{3}$ and let $c:V(D)
  \rightarrow [0,p)$ be an acyclic $p$-colouring of $D$.  We may
  assume $c(w)=0$.  We will show that $|c(v)|_p \ge 2-p$ for all $v
  \in V(D)\backslash \{w\}$.  

  Assume this was wrong.  Possibly replacing $c$ by $\tilde c := p-c
  \mod p$ this yields the existence of a vertex $v \in V(D)\backslash
  \{w\}$ such that $0 \le c(v) < 2-p$.  Let $u$ be the other vertex in
  the unique directed triangle containing $w$ and $v$. Let
  $m:=\frac{c(v)}{2} \in S_p$, then $S_p \subseteq [0,(\frac{p}{2}+m)
  \text{ mod }p]_p \cup [(\frac{p}{2}+m) \text{ mod }p, c(v)]_p$ and
  $|\frac{p}{2}+m -c(v)|_p = \frac{p}{2} + m < \frac{p+2-p}{2}=1$.
  Hence in any case $= \{c(w),c(u),c(x)\}$ is contained in an
  interval of length strictly smaller than $1$ contradicting $c$
  being an acyclic $p$-colouring.

  Thus, indeed $|c(v)|_p \ge 2-p>\frac{1}{3}$ for all outer vertices.
  Hence the image of the outer directed cycle under $c$ is contained
  in an open cyclic subinterval of length $p-\frac{2}{3}<1$, again
  contradicting the definition of an acyclic $p$-colouring. We
  conclude $\vec{\chi}^\ast(W_k) \ge \vec{\chi}^\ast(D) \ge
  \frac{5}{3}$ for this special orientation, proving the claims for
  the star dichromatic number.
  
  We now turn to  proving $\vec{\chi}_f(W_k) \leq
  \frac{3k-2}{2k-2}$. 
  Denote by $V^+,V^-$ a bipartition of the outer cycle of $W_k$. Note
  that $V^+ \cup \{w\}, V^-\cup\{w\}$ and all subsets of $V(W_k)
  \backslash \{w\}$ of size $k-1$ induce forests in $W_k$, hence also
  acyclic sets for any orientation of $W_k$.  We construct an instance
  of~\eqref{primal} by putting a weight of $\frac{1}{2}$ on each of
  $V^+ \cup \{w\}, V^- \cup \{w\}$, and a weight of $\frac{1}{2(k-1)}$
  on each of the $k$ subsets of $V(W_k) \backslash \{w\}$ of size
  $k-1$, all  other acyclic vertex sets receive a weight of $0$.
  We compute $\frac{1}{2}+\frac{1}{2} \geq 1$ for $w$ and $(k-1)
  \cdot \frac{1}{2(k-1)}+\frac{1}{2} \geq 1$ for each outer vertex.
  Hence, we have a feasible instance ~\eqref{primal}, verifying
  $\vec{\chi}_f(W_k) \leq \frac{1}{2}+\frac{1}{2}+
  \frac{k}{2(k-1)}=\frac{3k-2}{2k-2}$ as
  claimed. 

  Finally, we prove $\vec{\chi}^\ast(W_k) \ge \frac{3k-2}{2k-2}$ using
  the same special orientation $D$ of $W_k$ where the outer cycle is
  directed and the orientations of edges incident to $w$ are
  alternating in cyclic order. We construct a suitable instance of the
  dual program \eqref{dual}, defining $y_w:=\frac{k-2}{2k-2}$ and
  $y_v:=\frac{1}{k-1}$ for every outer vertex.  Let $A$ be a maximal
  acyclic set. If $w \notin A$, then, since the
  outer cycle is directed, $A=V(W_k) \backslash \{w,x\}$ for some
  outer vertex $x$. In this case, we verify
$$\sum_{v \in A}{y_v}=(k-1)\cdot\frac{1}{k-1}=1.$$
If $w \in A$, clearly $A$ contains at most one vertex of each directed triangle. Therefore \mbox{$|A \setminus \{w\}| \leq \frac{k}{2}$} and again we verify the restriction:
$$\sum_{v \in A}{y_v} \leq \frac{k}{2} \cdot \frac{1}{k-1}+\frac{k-2}{2k-2}=1.$$
Using linear programming duality we find $\vec{\chi}_f(D) \ge \sum_{v
  \in
  V(W_k)}{y_v}=\frac{k}{k-1}+\frac{k-2}{2k-2}=\frac{3k-2}{2k-2}$.
\end{itemize}
\end{proof}
Concerning fractional dichromatic numbers, Conjecture \ref{girth} would imply a tight upper bound of $\vec{\chi}_f(D) \leq \frac{g-1}{g-2}$ for planar digraphs of digirth $g$. In the following, we want to approach this upper bound by showing that indeed, $\vec{\chi}_f(D)$ tends to $1$ for planar digraphs of large digirth. This is not at all obvious, as it is known that when dropping the restriction of planarity, directed graphs with arbitrarily large digirth may have arbitrarily large dichromatic number at the same time (cf. \cite{twores}). In order to do so, we recall the following terminologies as well as a related famous max-min-principle, known as Lucchesi-Younger-Theorem:
\begin{definition}
\noindent
\begin{itemize}
\item A  \emph{clutter} is a set family with no members containing each other.
\item A subset of arcs in a digraph is called \emph{dijoin} if it intersects ever directed cut.
\item A subsets of arcs in a digraph is called \emph{feedback arc set} if it intersects every directed cycle.
\end{itemize}
\end{definition} 
\begin{theorem}[Lucchesi-Younger, cf. \cite{lucchesiyounger}]
Let $D$ be a digraph and $w:E(D) \rightarrow \mathbb{N}_0$ a weighting of the edges with non-negative integers. Then the minimal weight of a dijoin in $D$ equals the maximum number of (minimal) dicuts in $D$ so that every arc $a$ is contained in at most $w(a)$ of them. 
\end{theorem}
The terminology used in the following refers to \cite{corn}, especially Chapter 1.1. According to Definition 1.5 and Theorem 1.25 in \cite{corn}, the Lucchesi-Younger-Theorem means that the clutter of minimal directed cuts in any digraph admits the Max-Flow-Min-Cut-Property (MFMC). Consecutive application of Theorem 1.8 and Theorem 1.17 in \cite{corn}, where the latter is a theorem of Lehman (cf. \cite{lehman}), yields that the blocker of the clutter of minimal directed cuts, namely the clutter of minimal dijoins, is ideal. This means the following statement which was also pointed out in the article \cite{OPG} on Woodall's Conjecture in the Open Problem Garden:
\begin{theorem}
Let $D$ be a digraph and let $g$ denote the minimal size of a directed cut in $D$. Then there is $m \in \mathbb{N}$ and a collection of dijoins $J_1,\ldots,J_m$ equipped with a weighting $x_1,...,x_m \in \mathbb{R}_+$ such that $x_1+\ldots+x_m=g$ and for every arc $e \in E(D)$, we have $\sum_{i: e \in J_i}{x_i} \leq 1$. 
\end{theorem}
By considering planar digraphs and their directed duals, the dualities between minimal directed cuts and directed cycles as well as of dijoins and feedback arc sets yield:
\begin{corollary} \label{DeVosPlanar}
If $D$ is a planar digraph of digirth $g$, then there are $m$ feedback arc sets $F_1,\ldots,F_m \subseteq E(D)$ equipped with a weighting $x_1,\ldots,x_m \in \mathbb{R}_+$ so that $x_1+\ldots+x_m=g$ and for each edge $e \in E(D)$, $\sum_{i: e \in F_i}{x_i} \leq 1$.
\end{corollary}
From this we may now conclude an upper bound for the fractional dichromatic number which approaches $1$ for planar digraphs of large digirth.
\begin{theorem}
Let $g \ge 6$. Then for every planar digraph $D$ of digirth $g$ we have $\vec{\chi}_f(D) \leq \frac{g}{g-5}$.
\end{theorem}
\begin{proof}
Without loss of generality assume $D$ to be simple. Let $F_1,\ldots,F_m,x_1,\ldots,x_m$ be as given by Corollary \ref{DeVosPlanar}. Use the 5-degeneracy of the underlying graph $U(D)$ of $D$ to derive an ordering $v_1,\ldots,v_n, n:=|V(D)|$ of the vertices so that for each $i \in \{1,\ldots,n\}$, $v_i$ has degree at most $5$ in $G_i:=U(D)[v_1,\ldots,v_i]$. For each $v_i$, let $c(v_i)$ denote the set of $j \in \{1,...,m\}$ so that $v_i$ has an incident edge in $F_j \cap E(G_i)$. Then clearly $$\sum_{j \in c(v_i)}{x_j} \leq \sum_{e \in E_{G_i}(v)}{\sum_{j: e \in F_j}{x_j}} \leq \text{deg}_{G_i}(v_i) \leq 5$$ for each $v_i$. Furthermore, the vertex set $X_j:=\{x \in V(D)|j \notin c(x)\}$ is acyclic in $D$ for all $j=1,\ldots,m$: In any directed cycle $C$ in $D$ we find an arc contained in $F_j$, and thus, $j$ is contained in at least one of the $c$-sets of its end vertices. \par\noindent We now define an instance of the linear optimization program \ref{primal} defining $\vec{\chi}_f(D)$ according to $x_A=\frac{i_A}{g-5}$, where  $i_A=\sum_{j \in \{1,\ldots,m\}:A=X_j}{x_j}$ for each $A \in \mathcal{A}(D)$. Then those variables are non-negative and for each vertex $v$, we have
$$\sum_{A \in \mathcal{A}(D,v)}{i_A}=\sum_{j \in \{1,\ldots,m\}:v \in X_j}{x_j}=\sum_{j \notin c(v)}{x_j}=\sum_{j=1}^{m}{x_j}-\sum_{j \in c(v)}{x_j} \geq g-5.$$ Hence this is a legal instance proving $\vec{\chi}_f(D) \leq \sum_{A \in \mathcal{A}}{\frac{i_A}{g-5}}=\frac{\sum_{j=1}^{m}{x_j}}{g-5}=\frac{g}{g-5}$.
\end{proof}
\section{Conclusion and Some Open Problems}
The star dichromatic number of a digraph introduced and analysed in
this paper seems to share all desirable attributes of the competing
parameter from \cite{bokal}, the circular chromatic number. But, while
the star dichromatic number is always a lower bound for the circular
dichromatic number, it has the additional advantage that it is immune
to the existence of directed cuts, while the addition of a dominating
source makes the circular dichromatic number hit the ceiling. We
therefore believe that the parameter introduced in the present paper
yields a preferable generalization of the star chromatic
number of Vince to the directed case. This is also supported by the fact that it can be seen as oriented version of the circular vertex arboricity.

In the planar case it might be true that the star chromatic number
approaches 1 when the digirth increases (Conjecture \ref{girth}). Note
that this is impossible for $\vec{\chi}_c(D)$. It might be rewarding
to study in particular the case $g=4$ of Conjecture~\ref{girth}, e.g.,
orientations of planar triangulations without directed triangles,
as recently there has been substantial progress towards digraph
colourings of this class (\cite{mohar4}). 

Also, it would be interesting to determine the computational
complexity of decision problems of the form:
\\
Instance: A digraph $D$ (possibly from a certain class) and a real
number $p>1$. \\
Decide whether $\vec{\chi}^\ast(D) \leq p$.
\\
In \cite{bojannp} it was shown that corresponding decision problem for
the circular dichromatic number is NP-complete, even if restricted to
planar digraphs. We conjecture that the same should be true for the
star dichromatic number. This is true at least for all integers $p \in
\mathbb{N}, p\ge 2$, since in that case $\vec{\chi}^\ast(D) \leq p
\Leftrightarrow \vec{\chi}_c(D) \leq p$ for all digraphs $D$.\\

We want to conclude with an incomplete list of
other natural questions that remain unanswered in this paper:

\begin{itemize}
\item For any $g \ge 4$, is there a planar digraph of digirth $g$ with $\vec{\chi}^\ast(D)=\frac{g-1}{g-2}$?
\item Is there a meaningful characterization of digraphs with
  $\vec{\chi}^\ast(D)=\vec{\chi}_f(D)$? 
\item Which digraphs satisfy $\vec{\chi}^\ast(D)=\vec{\chi}(D)$ or,
  more generally,   $\vec{\chi}^\ast(D)=\vec{\chi}_c(D)$?
\item What about the star dichromatic number of tournaments?
\item Does the following statement hold true: For any $\varepsilon>0$, there is a $\delta=\delta(\varepsilon)>0$ such that any planar digraph $D$ with $\vec{\chi}_f(D) \leq 1+\delta$ has $\vec{\chi}^\ast(D) \leq 1+\varepsilon$ ?
\end{itemize}
\paragraph{Acknowledgements}
Part of this research was carried out during a stay of the second author at the Workshop Cycles and Colourings 2018, High Tatras, Slovakia. The collaborative, inspiring and friendly atmosphere during the workshop is gratefully acknowledged. The second author was supported by DFG-GRK 2434.
\bibliography{literaturmaster} 
\bibliographystyle{alpha}

\end{document}